\documentclass[12pt]{article}
\usepackage[]{amsmath,amssymb}
\usepackage{amscd}
\usepackage{latexsym}
\usepackage{cite}
\usepackage{amsthm}

\newtheorem{definition}{Definition}[section]
\newtheorem{theorem}[definition]{Theorem}
\newtheorem{lemma}[definition]{Lemma}
\newtheorem{corollary}[definition]{Corollary}

\newtheorem{example}[definition]{Example}
\newtheorem{conjecture}[definition]{Conjecture}
\newtheorem{problem}[definition]{Problem}
\newtheorem{note}[definition]{Note}

\newtheorem{proposition}[definition]{Proposition}

\begin{document}
\title{\bf 
A compact presentation for the \\
alternating central extension
of 
the
\\
positive part of
$U_q(\widehat{\mathfrak{sl}}_2)$ }
\author{
Paul Terwilliger 
}
\date{}

\maketitle
\begin{abstract}
This paper concerns the positive part $U^+_q$ of the quantum group  $U_q({\widehat{\mathfrak{sl}}}_2)$.
The algebra $U^+_q$ has a presentation involving two generators that satisfy the cubic $q$-Serre relations. We recently
introduced an algebra $\mathcal U^+_q$ called the alternating central extension of $U^+_q$. We presented $\mathcal U^+_q$
by generators and relations. The presentation is attractive, but the multitude of generators and relations makes
the presentation unwieldy. In this paper we obtain a presentation of $\mathcal U^+_q$
 that involves a small subset of the original set of generators and a very manageable set of relations. We call this presentation the
 compact presentation of $\mathcal U^+_q$. 

\bigskip

\noindent
{\bf Keywords}. $q$-Onsager algebra; $q$-Serre relations; $q$-shuffle algebra; tridiagonal pair.
\hfil\break
\noindent {\bf 2020 Mathematics Subject Classification}. 
Primary: 17B37. Secondary: 05E14, 81R50.

 \end{abstract}

\section{Introduction}
The algebra $U_q({\widehat{\mathfrak{sl}}}_2)$ is well known in representation theory \cite{charp} and statistical mechanics
\cite{JM}. This algebra has 
 a subalgebra $U^+_q$ called the positive 
part. The algebra $U^+_q$ has a presentation involving two generators (said to be standard) and two relations,
called the $q$-Serre relations. 
The presentation is given in Definition \ref{def:U} below.
\medskip

\noindent 
Our interest in $U^+_q$  is motivated by some applications to linear algebra and combinatorics;  these will be described shortly. Before going into detail, we have a comment about $q$.
In the applications, either $q$ is not a root of unity, or $q$ is a root of unity with exponent large enough to not interfere with the rest of the application.
To keep things simple, throughout the paper we will assume that $q$ is not a root of unity.
\medskip

\noindent
Our first application has to do with tridiagonal pairs \cite{TD00}.
 A  tridiagonal pair is roughly described as an ordered pair of diagonalizable linear maps on
a nonzero finite-dimensional vector space, that each act on the eigenspaces of the other one in a block-tridiagonal fashion \cite[Definition~1.1]{TD00}. There is a type of tridiagonal pair said to be $q$-geometric \cite[Definition~2.6]{nonnil};
for this type of tridiagonal pair the eigenvalues of each map form  a $q^2$-geometric progression. A finite-dimensional irreducible $U^+_q$-module on which the standard generators are not nilpotent,
is essentially the same thing as a tridiagonal pair of $q$-geometric type \cite[Theorem~2.7]{nonnil}; these $U^+_q$-modules are described in \cite[Section~1]{nonnil}. See \cite{bockting}, \cite{NomTerTB} for more background on tridiagonal pairs.
\medskip

\noindent
Our next application has to do with distance-regular graphs \cite{banIto}, \cite{BCN}, \cite{dkt}.
Consider a distance-regular graph $\Gamma$ that has diameter $d\geq 3$ and classical parameters $(d,b,\alpha, \beta)$  \cite[p.~193]{BCN}  with  $b=q^2$ and $\alpha = q^2-1$.
The condition on $\alpha$ implies that $\Gamma$ is formally self-dual in the sense of \cite[p.~49]{BCN}.
Let $A$ denote the adjacency matrix of $\Gamma$, and let $A^*$ denote the dual adjacency matrix with respect to any vertex of $\Gamma$ \cite[Section~7]{drg}.
Then by \cite[Lemma~9.4]{drg}, there exist complex numbers  $r, s, r^*, s^*$ with $r, r^*$ nonzero such that $rA+sI$, $r^* A^*+s^* I$ satisfy the
$q$-Serre relations. As mentioned in
\cite[Example~8.4]{drg},
 the above parameter restriction is satisfied by the
bilinear forms graph \cite[p.~280]{BCN},
the alternating forms graph \cite[p.~282]{BCN},
the Hermitean forms graph \cite[p.~285]{BCN},
the quadratic forms graph \cite[p.~290]{BCN},
the affine $E_6$ graph \cite[p.~340]{BCN},
and the extended ternary Golay code graph \cite[p.~359]{BCN}.
\medskip

\noindent Our next application has to do with uniform posets \cite{miklavic}, \cite{uniform}. Let ${\rm GF}(b)$ denote a finite field with $b$ elements, and 
let $N, M$ denote positive integers.
  Let $H$ denote a vector space over ${\rm GF}(b)$ that has dimension $N+M$. Let $h$ denote a subspace of $H$ with dimension $M$. Let $P$ denote the set of subspaces of $H$ that have zero  intersection with $h$.
  For $x,y \in P$ define $x \leq y$ whenever $x \subseteq y$. The relation $\leq $ is a partial order on $P$, and the poset $P$ is ranked with rank $N$.
  The poset $P$ is called an attenuated space poset, and denoted by $A_b(N,M)$ \cite{wenLiu}, \cite[Example~3.1]{uniform}.
  By \cite[Theorem~3.2]{uniform} the poset  $A_b(N, M)$  is uniform in the sense of \cite[Definition~2.2]{uniform}.
  It is shown in \cite[Lemma~3.3]{wenLiu} that for $A_b(N,M)$ the raising matrix $R$ and the lowering matrix $L$ satisfy the $q$-Serre relations, 
  provided that $b=q^2$.

\medskip

\noindent Our last application has to do with $q$-shuffle algebras. Let $\mathbb F$ denote a field, and let $x, y$ denote noncommuting indeterminates.
Let $V$ denote the free associative $\mathbb F$-algebra with generators $x, y$.  By a letter in $V$ we mean $x$ or $y$. For an integer $n\geq 0$, by a word of length $n$ in $V$
we mean a product of letters $v_1 v_2 \cdots v_n$. The words in $V$ form a basis for the vector space $V$.
In
\cite{rosso1, rosso} M. Rosso introduced
an algebra structure on $V$, called the
$q$-shuffle algebra.
For letters $u, v$ their
$q$-shuffle product is
$u\star v = uv+q^{\langle u,v\rangle }vu$, where
$\langle u,v\rangle =2$
(resp. $\langle u,v\rangle =-2$)
if $u=v$ (resp.
 $u\not=v$).
 By \cite[Theorem~13]{rosso1},  in the $q$-shuffle algebra $V$ the elements $x$, $y$ satisfy the $q$-Serre relations.
 Consequently there exists an algebra homomorphism  $\natural$
from $U^+_q$ into the $q$-shuffle algebra
$V$, that sends the standard generators of $U^+_q$ to $x$, $y$.
By \cite[Theorem~15]{rosso} the map $\natural $ is injective. 
\medskip

\noindent Next we recall the alternating elements in
 $U^+_q$ \cite{alternating}. 
 Let  $v_1v_2\cdots v_n$ denote a word in $V$. This word is called alternating whenever $n\geq 1$ and $v_{i-1} \not=v_i$
 for $2 \leq i \leq n$. Thus the alternating words have the form $\cdots xyxy \cdots$.
 The alternating words are displayed below:
\begin{align*}
& x, \qquad  xyx, \qquad xyxyx, \qquad xyxyxyx, \qquad \ldots
\\
&y, \qquad yxy, \qquad yxyxy,  \qquad yxyxyxy, \qquad \ldots
\\
&yx, \qquad yxyx,  \qquad yxyxyx, \qquad yxyxyxyx, \qquad  \ldots
\\
& xy, \qquad  
xyxy,\qquad  xyxyxy, \qquad xyxyxyxy, \qquad \ldots
\end{align*}
\noindent 
By \cite[Theorem~8.3]{alternating} each alternating word is contained in the image of $\natural$. An element of $U^+_q$
is called alternating whenever it is the 
$\natural$-preimage of an alternating word. For example, the standard generators of $U^+_q$ are alternating because they are the
$\natural$-preimages of the alternating words $x,y$. 
It is shown in \cite[Lemma~5.12]{alternating} that for each row in the above display, the corresponding alternating elements mutually commute. 
A naming scheme for alternating elements is introduced in  \cite[Definition~5.2]{alternating}.
\medskip


\noindent Next we recall the alternating central extension of $U^+_q$ \cite{altCE}. In \cite{alternating} we 
displayed two types of relations among
the alternating elements of $U^+_q$; the first type is \cite[Propositions~5.7,~5.10,~5.11]{alternating} and the second type is
\cite[Propositions~6.3,~8.1]{alternating}. 
The relations in \cite[Proposition~5.11]{alternating} are redundant;
they follow from the relations in \cite[Propositions~5.7,~5.10]{alternating} as pointed out in \cite[Propositions~3.1, 3.2]{baspp} and \cite[Remark~2.5]{basFMA};
see also Corollary \ref{cor:redundant} below.
The relations in \cite[Proposition 6.3]{alternating} are also redundant; they follow from
the relations in \cite[Propositions~5.7,~5.10]{alternating} as shown in the proof of \cite[Proposition 6.3]{alternating}.
By \cite[Lemma~8.4]{alternating} and the previous comments,  the algebra $U^+_q$ is presented by
its alternating elements and the relations in \cite[Propositions~5.7,~5.10,~8.1]{alternating}. 
For this presentation it is natural to ask what happens if the
 relations in \cite[Proposition~8.1]{alternating} are removed. To answer this question, in \cite[Definition~3.1]{altCE}
we defined an algebra $\mathcal U^+_q$ by generators and relations in the following way. The generators, said to be alternating,
are in bijection with the alternating elements of $U^+_q$. The relations are the ones in \cite[Propositions~5.7, 5.10]{alternating}. By construction there
exists a surjective algebra homomorphism  $\mathcal U^+_q\to U^+_q$ that sends each alternating generator of $\mathcal U^+_q$
to the corresponding alternating element of $U^+_q$. In \cite[Lemma~3.6, Theorem~5.17]{altCE} we adjusted this homomorphism to get an algebra isomorphism
$\mathcal U^+_q \to U^+_q \otimes \mathbb F \lbrack z_1, z_2, \ldots \rbrack$, where
$\lbrace z_n\rbrace_{n=1}^\infty $ are mutually commuting indeterminates. By \cite[Theorem~10.2]{altCE} the alternating generators  form a PBW basis
for $\mathcal U^+_q$. The algebra $\mathcal U^+_q$ is called the alternating central extension of $U^+_q$.
\medskip

\noindent 
We mentioned above that the algebra  $\mathcal U^+_q$ is presented by its alternating generators and the relations in \cite[Propositions~5.7,~5.10]{alternating}. 
This presentation is attractive, but the
 multitude of generators and relations makes the presentation unwieldy. In this paper we obtain a presentation of $\mathcal U^+_q$
 that involves a small subset of the original set of generators and a very manageable set of relations. This presentation is given in Definition \ref{def:CE}
  below; we call it the
 compact presentation of $\mathcal U^+_q$. At first glance, it is not clear that the algebra presented in Definition \ref{def:CE} is equal to $\mathcal U^+_q$.
 So we denote by $\mathcal U$ the algebra presented in Definition \ref{def:CE}, and eventually prove that $\mathcal U = \mathcal U^+_q$.
 After this result is established, we describe some features of $\mathcal U^+_q$ that are illuminated by the presentation in Definition \ref{def:CE}. 
\medskip

\noindent Our investigation of $\mathcal U^+_q$ is inspired by some recent developments in statistical mechanics, concerning the $q$-Onsager algebra $O_q$.
In \cite{BK05} Baseilhac and Koizumi introduce a current algebra $\mathcal A_q$ for $ O_q$, in order to solve boundary integrable systems with hidden symmetries.
In \cite[Definition~3.1]{basnc} Baseilhac and Shigechi give a presentation of $\mathcal A_q$ by generators and relations. This presentation and the discussion in \cite[Section~4]{basnc} suggest that $\mathcal A_q$
is related to $O_q$ in roughly the same way that $\mathcal U^+_q$
is related to $U^+_q$. 
The relationship between $\mathcal A_q$ and $ O_q$ was conjectured in
\cite[Conjectures~1,~2]{basBel} and
 \cite[Conjectures~4.5,~4.6,~4.8]{z2z2z2}, before being settled in \cite[Theorems~9.14, 10.2, 10.3, 10.4]{pbwqO}.
The articles 
\cite{bas2,
bas1,
basXXZ,
basBel,
BK05,
bas4,
basKoi,
BK,
basnc}
contain background information on $ O_q$ and $\mathcal A_q$.
\medskip

\noindent Earlier in this section, we indicated how  $U^+_q$ has applications to tridiagonal pairs, distance-regular graphs, and uniform posets. Possibly $\mathcal U^+_q$
appears in these applications,  and this  possibility should be investigated in the future.
\medskip

\noindent This paper is organized as follows. In Section 2 we review some facts about $U^+_q$.
In Section 3, we introduce the algebra $\mathcal U$ and give an algebra homomorphism $U^+_q \to \mathcal U$.
In Section 4, we introduce the alternating generators for $\mathcal U$ and establish some formulas involving these generators.
In Sections 5, 6 we use these formulas and generating functions to show that the alternating generators for $\mathcal U$ satisfy the relations in
\cite[Propositions~5.7,~5.10]{alternating}. Using this result, we prove that $\mathcal  U= \mathcal U^+_q$. Theorem \ref{thm:Fin} and Corollary \ref{cor:altExt} are the main results of the paper.
In Section 7 we describe some features of $\mathcal U^+_q$ that are illuminated by the presentation in Definition \ref{def:CE}. Appendix A contains a list of relations involving the generating
functions from Section 5.

\section{The algebra $U^+_q$} We now begin our formal argument.
For the rest of the paper, the following notational conventions are in effect.
Recall the natural numbers 
$\mathbb N =\lbrace 0,1,2,\ldots \rbrace$.
 Let $\mathbb F$ denote a field. Every vector space and tensor product 
 mentioned is over $\mathbb F$. Every algebra mentioned is associative, over $\mathbb F$, and 
has a multiplicative identity. Fix a nonzero $q \in \mathbb F$
that is not a root of unity.
Recall the notation
\begin{eqnarray*}
\lbrack n\rbrack_q = \frac{q^n-q^{-n}}{q-q^{-1}}
\qquad \qquad n \in \mathbb N.
\end{eqnarray*}
\noindent For elements $X, Y$ in any algebra, define their
commutator and $q$-commutator by 
\begin{align*}
\lbrack X, Y \rbrack = XY-YX, \qquad \qquad
\lbrack X, Y \rbrack_q = q XY- q^{-1}YX.
\end{align*}
\noindent Note that 
\begin{align*}
\lbrack X, \lbrack X, \lbrack X, Y\rbrack_q \rbrack_{q^{-1}} \rbrack
= 
X^3Y-\lbrack 3\rbrack_q X^2YX+ 
\lbrack 3\rbrack_q XYX^2 -YX^3.
\end{align*}

\begin{definition} \label{def:U}
\rm (See \cite[Corollary~3.2.6]{lusztig}.)
Define the algebra $U^+_q$ by generators $W_0$, $W_1$ and relations
\begin{align}
\label{eq:qSerre1}
\lbrack W_0, \lbrack W_0, \lbrack W_0, W_1\rbrack_q \rbrack_{q^{-1}} \rbrack =0,
\qquad \qquad 
\lbrack W_1, \lbrack W_1, \lbrack W_1, W_0\rbrack_q \rbrack_{q^{-1}}\rbrack =0.
\end{align}
We call $U^+_q$ the {\it positive part of $U_q({\widehat{\mathfrak{sl}}}_2)$}.
The generators $W_0, W_1$ are called {\it standard}.
The relations (\ref{eq:qSerre1}) are called the {\it $q$-Serre relations}.
\end{definition}
\noindent We will use the following concept.
\begin{definition}\rm 
(See \cite[p.~299]{damiani}.)
Let $ \mathcal A$ denote an algebra. A {\it Poincar\'e-Birkhoff-Witt} (or {\it PBW}) basis for $\mathcal A$
consists of a subset $\Omega \subseteq \mathcal A$ and a linear order $<$ on $\Omega$
such that the following is a basis for the vector space $\mathcal A$:
\begin{align*}
a_1 a_2 \cdots a_n \qquad n \in \mathbb N, \qquad a_1, a_2, \ldots, a_n \in \Omega, \qquad
a_1 \leq a_2 \leq \cdots \leq a_n.
\end{align*}
We interpret the empty product as the multiplicative identity in $\mathcal A$.
\end{definition}

\noindent In \cite[p.~299]{damiani} Damiani obtains a PBW basis for $U^+_q$ that involves some elements
\begin{align}
\lbrace E_{n \delta+ \alpha_0} \rbrace_{n=0}^\infty,
\qquad \quad 
\lbrace E_{n \delta+ \alpha_1} \rbrace_{n=0}^\infty,
\qquad \quad 
\lbrace E_{n \delta} \rbrace_{n=1}^\infty.
\label{eq:Upbw}
\end{align}
These elements are defined recursively as follows:
\begin{align}
E_{\alpha_0}= W_0, \qquad \quad
E_{\alpha_1} = W_1, \qquad \quad
E_\delta = q^{-2} W_1 W_0 - W_0 W_1
\label{eq:dam1}
\end{align}
and for $n\geq 1$,
\begin{align}
\label{eq:dam2}
E_{n \delta + \alpha_0} &= \frac{\lbrack E_\delta, E_{(n-1)\delta + \alpha_0}\rbrack }{q+q^{-1}},
\qquad \qquad 
E_{n \delta + \alpha_1} = \frac{\lbrack  E_{(n-1)\delta + \alpha_1}, E_\delta\rbrack}{q+q^{-1}},
\\
&
E_{n \delta} = q^{-2} E_{(n-1)\delta + \alpha_1} W_0 - W_0 E_{(n-1)\delta + \alpha_1}.
\label{eq:dam3}
\end{align}

\begin{proposition}
\label{prop:damiani} {\rm (See \cite[p.~308]{damiani}.)}
 A PBW basis for $U^+_q$ is obtained by the elements {\rm (\ref{eq:Upbw})} in the linear
order
\begin{align*}
E_{\alpha_0} <
E_{\delta+ \alpha_0} < 
E_{2 \delta + \alpha_0} < 
\cdots < 
E_\delta < 
E_{2\delta} < 
E_{3\delta} < 
\cdots < 
E_{2\delta+ \alpha_1} <
E_{\delta+ \alpha_1} < 
E_{\alpha_1}.
\end{align*}
\end{proposition}
\noindent The elements 
(\ref{eq:Upbw}) satisfy many relations \cite{damiani}. We mention a few for later use.

\begin{lemma}
\label{lem:com3}
{\rm (See \cite[p.~300]{damiani}.)}
For  $i,j \in \mathbb N$ with $i>j$ the following hold in $U^+_q$.
\begin{enumerate}
\item[\rm (i)] Assume that $i-j=2r+1$ is odd. Then
\begin{align*}
&
E_{i\delta+\alpha_0}
E_{j\delta+\alpha_0}
=
q^{-2}
E_{j\delta+\alpha_0}
E_{i\delta+\alpha_0}
-
(q^2-q^{-2})\sum_{\ell=1}^{r}
q^{-2\ell}
E_{(j+\ell) \delta+\alpha_0}
E_{(i-\ell) \delta+\alpha_0},
\\
&
E_{j\delta+\alpha_1}
E_{i\delta+\alpha_1} =
q^{-2}
E_{i\delta+\alpha_1 }
E_{j\delta+\alpha_1 }
-
(q^2-q^{-2})\sum_{\ell=1}^{r}
q^{-2\ell}
E_{(i-\ell) \delta+\alpha_1}
E_{(j+\ell) \delta+\alpha_1}.
\end{align*}
\item[\rm (ii)] Assume that $i-j=2r$ is even. Then
\begin{align*}
E_{i\delta+\alpha_0}
E_{j\delta+\alpha_0}
 =
q^{-2}
E_{j\delta+\alpha_0}
&E_{i\delta+\alpha_0}
-
q^{j-i+1} (q-q^{-1}) E^2_{(r+j)\delta+\alpha_0}
\\
&-\;
(q^2-q^{-2})\sum_{\ell=1}^{r-1}
q^{-2\ell}
E_{(j+\ell) \delta+\alpha_0}
E_{(i-\ell) \delta+\alpha_0},
\\
E_{j\delta+\alpha_1}
E_{i\delta+\alpha_1} =
q^{-2}
E_{i\delta+\alpha_1 }
&
E_{j\delta+\alpha_1 }
-
q^{j-i+1} (q-q^{-1}) E^2_{(r+j)\delta+\alpha_1}
\\
&-\;
(q^2-q^{-2})\sum_{\ell=1}^{r-1}
q^{-2\ell}
E_{(i-\ell) \delta+\alpha_1}
E_{(j+\ell) \delta+\alpha_1}.
\end{align*}
\end{enumerate}
\end{lemma}

\begin{lemma} The following {\rm (i)--(iii)}  hold in $U^+_q$.
\begin{enumerate}
\item[\rm (i)] {\rm (See \cite[p.~307]{damiani}.)} For positive $i,j \in \mathbb N$,
\begin{align}
E_{i \delta} E_{j \delta} = E_{j \delta} E_{i \delta}.
\label{eq: mutCom}
\end{align}
\item[\rm (ii)] {\rm (See \cite[p.~307]{damiani}.)} For $i,j\in \mathbb N$,
\begin{align}
\lbrack E_{i \delta + \alpha_0}, E_{j \delta + \alpha_1} \rbrack_q = - q E_{(i+j+1) \delta}.
\label{eq:qcomE}
\end{align}
\item[\rm (iii)] For $i \in \mathbb N$,
\begin{align}
\label{eq:WEcom}
\frac{ \lbrack W_0, E_{i\delta + \alpha_0} \rbrack_q } {q-q^{-1}} &= \sum_{\ell = 0}^i E_{\ell \delta + \alpha_0} E_{(i-\ell) \delta + \alpha_0},
\\
\frac{\lbrack E_{i \delta + \alpha_1}, W_1\rbrack_q }{ q-q^{-1}} &=  \sum_{\ell=0}^i E_{(i-\ell) \delta + \alpha_1} E_{\ell \delta + \alpha_1}.
\label{eq:WEcom2}
\end{align}
\end{enumerate}
\end{lemma}
\begin{proof} (iii) To verify \eqref{eq:WEcom} and \eqref{eq:WEcom2}, use Lemma \ref{lem:com3} to write each term in the PBW basis for $U^+_q$ from Proposition \ref{prop:damiani}.
We give the details for \eqref{eq:WEcom}. Referring to \eqref{eq:WEcom}, let $\Delta$ denote the right-hand side minus the left-hand side. We show that $\Delta=0$.
This is quickly verified for $i=0$, so assume that $i\geq 1$. 
For $i$ even (resp. $i$ odd) write $i=2r$ (resp. $i=2r+1$).
Using Lemma \ref{lem:com3} we obtain $\Delta = \sum_{\ell =0}^r \alpha_\ell E_{\ell \delta+\alpha_0} E_{(i-\ell)\delta+\alpha_0}$, where
for $i$ even,
\begin{align*}
\alpha_0 & = 1+q^{-2} -\frac{q}{q-q^{-1}} + \frac{q^{-3}}{q-q^{-1}} ,
\\
\alpha_\ell &= 1+q^{-2} -(q^2-q^{-2})\sum_{k=1}^{\ell} q^{-2k} -(q+q^{-1})q^{-2\ell-1} \qquad (1 \leq \ell \leq r-1),
\\
\alpha_r &= 1-(q-q^{-1})\sum_{k=1}^r q^{1-2k} -q^{-i}
\end{align*}
and for $i$ odd,
\begin{align*}
\alpha_0 & = 1+q^{-2} -\frac{q}{q-q^{-1}} + \frac{q^{-3}}{q-q^{-1}} ,
\\
\alpha_\ell &= 1+q^{-2} -(q^2-q^{-2})\sum_{k=1}^{\ell} q^{-2k} -(q+q^{-1})q^{-2\ell-1} \qquad (1 \leq \ell \leq r).
\end{align*}
\noindent For either case $\alpha_\ell=0$ for $0 \leq \ell \leq r$, so $\Delta=0$. We have verified \eqref{eq:WEcom}. For  \eqref{eq:WEcom2} the details are
similar, and omitted.
\end{proof}

\section{An extension of $U^+_q$}

In this section we introduce the algebra $\mathcal U$. In Section 6 we will show that
$\mathcal U$ coincides with the alternating central extension $\mathcal U^+_q$  of $U^+_q$.
\newpage
\begin{definition}\rm
\label{def:CE}
Define the algebra $\mathcal U$ by generators $W_0$, $W_1$, $\lbrace \tilde G_{k+1} \rbrace_{k \in \mathbb N}$ and  relations
\begin{enumerate}
\item[\rm (i)] 
$ \lbrack W_0, \lbrack W_0, \lbrack W_0, W_1\rbrack_q \rbrack_{q^{-1}} \rbrack =0$,
\item[\rm (ii)] 
$ \lbrack W_1, \lbrack W_1, \lbrack W_1, W_0\rbrack_q \rbrack_{q^{-1}} \rbrack =0$,
\item[\rm (iii)] $ {\displaystyle {
       \lbrack \tilde G_1, W_1 \rbrack = q \frac{\lbrack \lbrack W_0, W_1\rbrack_q, W_1 \rbrack } { q^2-q^{-2}},
       }} $
       \item[\rm (iv)] $ {\displaystyle {
       \lbrack  W_0, \tilde G_1 \rbrack = q \frac{\lbrack W_0, \lbrack W_0, W_1\rbrack_q \rbrack } { q^2-q^{-2}},
       }} $
 \item[\rm (v)] for $k\geq 1$,
 \begin{align*}
 \lbrack \tilde G_{k+1}, W_1 \rbrack = \frac{ \lbrack \lbrack \lbrack \tilde G_k, W_0 \rbrack_q, W_1 \rbrack_q, W_1 \rbrack }{ (1-q^{-2})(q^2-q^{-2}) },
 \end{align*}
 \item[\rm (vi)] for $k\geq 1$,
 \begin{align*}
 \lbrack W_0, \tilde G_{k+1} \rbrack = \frac{ \lbrack W_0, \lbrack W_0, \lbrack W_1, \tilde G_k \rbrack_q \rbrack_q \rbrack}{(1-q^{-2})(q^2-q^{-2})},
 \end{align*}
 \item[\rm (vii)] for $k, \ell \in \mathbb N$,
 \begin{align*}
 \lbrack \tilde G_{k+1}, \tilde G_{\ell+1} \rbrack = 0.
 \end{align*}
       \end{enumerate}
       For notational convenience define $\tilde G_0=1$.
\end{definition}
\begin{note}\label{note:k}
 \rm Referring to Definition \ref{def:CE}, the relation (iii) (resp. (iv)) is obtained from (v) (resp. (vi)) by setting $k=0$.
\end{note}
\begin{lemma} \label{lem:flat}
There exists a unique algebra homomorphism $\flat: U^+_q \to \mathcal U$ that sends $W_0 \mapsto W_0$ and
$W_1 \mapsto W_1$. 
\end{lemma}
\begin{proof} Compare 
Definitions \ref{def:U}, \ref{def:CE}.
\end{proof}
 
 \noindent  In Corollary \ref{cor:inj} we will show that $\flat $ is injective. Let  $\langle W_0, W_1 \rangle $ denote the subalgebra of $\mathcal U$ generated by $W_0, W_1$.
Of course $\langle W_0, W_1\rangle $ is the $\flat $-image of $U^+_q$. For the elements 
 (\ref{eq:Upbw})
of $U^+_q$, the same notation will be used for their $\flat $-images in $\langle W_0, W_1\rangle$.

%

\section{Augmenting the generating set for $\mathcal U$}

\noindent 
 Some of the relations in Definition \ref{def:CE} are nonlinear. Our next goal is to  linearize the  relations by adding more generators.
\begin{definition}
\label{def:wwg} 
\rm 
We define some elements in $\mathcal U$ as follows. For $k \in \mathbb N$,
\begin{align}
W_{-k} &= \frac{ \lbrack \tilde G_k, W_0 \rbrack_q}{ q-q^{-1}},
\label{eq:Wmk}
\\
W_{k+1} &= \frac{ \lbrack W_1, \tilde G_k \rbrack_q }{q-q^{-1}},
\label{eq:Wkp1}
\\
G_{k+1} &= \tilde G_{k+1} + \frac{\lbrack W_1, W_{-k} \rbrack}{1-q^{-2}}.
\label{eq:Gkp1}
\end{align}
For notational convenience define $G_0=1$.
\end{definition}
\begin{lemma} For $k \in \mathbb N$ the following hold in $\mathcal U$:
\begin{align*}
\tilde G_k W_0 &= q^{-2} W_0 \tilde G_k + (1-q^{-2}) W_{-k},
\\
\tilde G_k W_1 &= q^{2} W_1 \tilde G_k + (1-q^{2}) W_{k+1}.
\end{align*}
\end{lemma}
\begin{proof} These are reformulations of (\ref{eq:Wmk}) and (\ref{eq:Wkp1}).
\end{proof}

\noindent  The following is a generating set for $\mathcal U$:
\begin{align}
\lbrace W_{-k} \rbrace_{k \in \mathbb N}, \qquad
\lbrace W_{k+1} \rbrace_{k \in \mathbb N}, \qquad
\lbrace G_{k+1} \rbrace_{k \in \mathbb N}, \qquad
\lbrace \tilde G_{k+1} \rbrace_{k \in \mathbb N}.
\label{eq:WWGG}
\end{align}
The elements of this set will be called {\it alternating}.
We seek a presentation of $\mathcal U$, that has the above generating set and all relations linear. We will
obtain this presentation in Theorem \ref{thm:Fin}.
\medskip

\noindent Next we obtain some formulas that will help us prove Theorem \ref{thm:Fin}.  We will show that for $n \in \mathbb N$,
\begin{align}
\label{eq:X}
W_{n+1} &= \sum_{k=0}^n \frac{ E_{k\delta + \alpha_1} \tilde G_{n-k} (-1)^k q^k }{ (q-q^{-1})^{2k}},
\\
\label{eq:Y}
W_{-n} &= \sum_{k=0}^n \frac{ E_{k\delta + \alpha_0} \tilde G_{n-k} (-1)^k q^{3k} }{ (q-q^{-1})^{2k}}.
\end{align}

\noindent We will prove (\ref{eq:X}), (\ref{eq:Y}) by induction on $n$. 
Note that (\ref{eq:X}), (\ref{eq:Y}) hold for $n=0$,
since $W_1=E_{\alpha_1}$ and $W_0= E_{\alpha_0}$.
 We will give the main induction argument after a few lemmas.
For the rest of this section $k$ and $\ell$ are understood to be in $\mathbb N$.

\begin{lemma} 
\label{lem:step1}
Pick $n \in \mathbb N$, and assume that {\rm (\ref{eq:X})}, {\rm (\ref{eq:Y})} hold for $n, n-1, \ldots, 1,0$. Then
\begin{align}
\lbrack W_0, W_{n+1} \rbrack = \lbrack W_{-n}, W_1 \rbrack.
\label{eq:step1}
\end{align}
\end{lemma}
\begin{proof}  The commutator $\lbrack W_0, W_{n+1} \rbrack$ is equal to
\begin{align*}
& W_0 W_{n+1} - W_{n+1} W_0
\\
&= \sum_{k=0}^n \frac{W_0 E_{k\delta + \alpha_1} \tilde G_{n-k} (-1)^k q^k }{ (q-q^{-1})^{2k}}
- \sum_{k=0}^n \frac{ E_{k\delta + \alpha_1} \tilde G_{n-k} W_0 (-1)^k q^k}{(q-q^{-1})^{2k}}
\\
&= \sum_{k=0}^n \frac{W_0 E_{k\delta + \alpha_1} \tilde G_{n-k} (-1)^k q^k }{ (q-q^{-1})^{2k}}
- \sum_{k=0}^n \frac{ E_{k\delta + \alpha_1} \bigl(q^{-2} W_0 \tilde G_{n-k} + (1-q^{-2})W_{k-n} \bigr)(-1)^k q^k}{(q-q^{-1})^{2k}}
\\
&= \sum_{k=0}^n \frac{\bigl( W_0 E_{k\delta + \alpha_1}-q^{-2} E_{k\delta+\alpha_1} W_0\bigr) \tilde G_{n-k} (-1)^k q^k }{ (q-q^{-1})^{2k}}
- \sum_{k=0}^n \frac{ E_{k\delta + \alpha_1} W_{k-n} (-1)^k q^{k-1}}{(q-q^{-1})^{2k-1}}
\\
&= -\sum_{k=0}^n \frac{E_{(k+1)\delta} \tilde G_{n-k} (-1)^k q^k }{ (q-q^{-1})^{2k}}
- \sum_{k=0}^n \frac{ E_{k\delta + \alpha_1} W_{k-n} (-1)^k q^{k-1}}{(q-q^{-1})^{2k-1}}
\\
&= -\sum_{k=0}^n \frac{E_{(k+1)\delta} \tilde G_{n-k} (-1)^k q^k }{ (q-q^{-1})^{2k}}
- \sum_{k=0}^n \frac{ E_{k\delta + \alpha_1} (-1)^k q^{k-1}}{(q-q^{-1})^{2k-1}}\sum_{\ell=0}^{n-k} \frac{ E_{\ell \delta + \alpha_0} \tilde G_{n-k-\ell} (-1)^\ell q^{3\ell}}{(q-q^{-1})^{2\ell}}
\\
&= -\sum_{p=0}^n \frac{E_{(p+1)\delta} \tilde G_{n-p} (-1)^p q^p }{ (q-q^{-1})^{2p}}
- \sum_{p=0}^n \Biggl( \sum_{k+\ell=p}
 q^{2\ell}  E_{k\delta + \alpha_1} E_{\ell \delta + \alpha_0} \Biggr) \frac{ \tilde G_{n-p} (-1)^p q^{p-1} }{(q-q^{-1})^{2p-1}}.
\end{align*}
\noindent The commutator $\lbrack W_{-n}, W_1 \rbrack$ is equal to
\begin{align*}
& W_{-n} W_1 - W_1 W_{-n}
\\
&= \sum_{k=0}^n \frac{ E_{k\delta + \alpha_0} \tilde G_{n-k} W_1 (-1)^k q^{3k} }{ (q-q^{-1})^{2k}}
- \sum_{k=0}^n \frac{ W_1 E_{k\delta + \alpha_0} \tilde G_{n-k} (-1)^k q^{3k} }{ (q-q^{-1})^{2k}}
\\
&= \sum_{k=0}^n \frac{ E_{k\delta + \alpha_0} \bigl( q^2 W_1 \tilde G_{n-k} + (1-q^2) W_{n-k+1} \bigr) (-1)^k q^{3k} }{ (q-q^{-1})^{2k}}
- \sum_{k=0}^n \frac{ W_1 E_{k\delta + \alpha_0} \tilde G_{n-k} (-1)^k q^{3k} }{ (q-q^{-1})^{2k}}
\\
&= \sum_{k=0}^n \frac{ \bigl( q^2 E_{k\delta + \alpha_0} W_1 -W_1 E_{k\delta + \alpha_0}  \bigr) \tilde G_{n-k} (-1)^k q^{3k} }{ (q-q^{-1})^{2k}}
- \sum_{k=0}^n \frac{ E_{k\delta + \alpha_0} W_{n-k+1} (-1)^k q^{3k+1} }{ (q-q^{-1})^{2k-1}}
\\
&= -\sum_{k=0}^n \frac{ E_{(k+1)\delta} \tilde G_{n-k} (-1)^k q^{3k+2} }{ (q-q^{-1})^{2k}}
- \sum_{k=0}^n \frac{ E_{k\delta + \alpha_0} W_{n-k+1} (-1)^k q^{3k+1} }{ (q-q^{-1})^{2k-1}}
\\
&= -\sum_{k=0}^n \frac{ E_{(k+1)\delta} \tilde G_{n-k} (-1)^k q^{3k+2}}{ (q-q^{-1})^{2k}}
-
\sum_{k=0}^n \frac{ E_{k\delta + \alpha_0} (-1)^k q^{3k+1}}{(q-q^{-1})^{2k-1}} \sum_{\ell=0}^{n-k} \frac{ E_{\ell \delta + \alpha_1} \tilde G_{n-k-\ell} (-1)^\ell q^\ell }{ (q-q^{-1})^{2\ell}}
\\
&
= -\sum_{p=0}^n \frac{ E_{(p+1)\delta} \tilde G_{n-p} (-1)^p q^{3p+2}}{ (q-q^{-1})^{2p}}
- \sum_{p=0}^n \Biggl( \sum_{k+\ell=p} 
 q^{2k}  E_{k\delta + \alpha_0} E_{\ell \delta + \alpha_1} \Biggr) \frac{ \tilde G_{n-p} (-1)^p q^{p+1} }{(q-q^{-1})^{2p-1}}.
\end{align*}
\noindent By these comments
\begin{align*} 
 \lbrack W_{-n}, W_1 \rbrack -
\lbrack W_0, W_{n+1} \rbrack
 = \sum_{p=0}^n \frac{C_p \tilde G_{n-p} (-1)^p q^p}{(q-q^{-1})^{2p}},
\end{align*}
where for $0 \leq p \leq n$,
\begin{align*}
C_p &=  E_{(p+1)\delta} +
 q^{-1}(q-q^{-1}) \sum_{k+\ell=p} q^{2\ell}  E_{k \delta+ \alpha_1} E_{\ell \delta + \alpha_0}\\
 &\qquad \qquad \qquad  - q^{2p+2} E_{(p+1)\delta}
-q(q-q^{-1}) \sum_{k+\ell=p} q^{2k} E_{k \delta + \alpha_0} E_{\ell \delta + \alpha_1}
\\
&=  (1-q^{2p+2}) E_{(p+1)\delta} - (1-q^2)\sum_{k+\ell=p} q^{2\ell} \bigl(q^{-2} E_{k \delta+ \alpha_1} E_{\ell \delta + \alpha_0}-
E_{\ell \delta + \alpha_0} E_{k \delta + \alpha_1} \bigr)
\\
 &=  (1-q^{2p+2})E_{(p+1)\delta} - (1-q^2)\sum_{k+\ell=p} q^{2\ell} E_{(p+1)\delta}
 \\
 &=  \biggl(1-q^{2p+2} - (1-q^2)\sum_{\ell=0}^p q^{2\ell} \biggr) E_{(p+1) \delta}
 \\
 &= 0.
 \end{align*}
 The result follows.
\end{proof} 

\begin{lemma} \label{lem:step2}
Pick $n \in \mathbb N$, and assume that {\rm (\ref{eq:X})}, {\rm (\ref{eq:Y})} hold for $n, n-1, \ldots, 1,0$. Then
\begin{align}
\lbrack \tilde G_n, E_\delta \rbrack = 0.
\label{eq:step2}
\end{align}
\end{lemma}
\begin{proof} Using Lemma \ref{lem:step1},
\begin{align*}
0 &=
(q-q^{-1}) \bigl( \lbrack W_{-n}, W_1 \rbrack - \lbrack W_0, W_{n+1} \rbrack \bigr)
\\
&= \lbrack \lbrack \tilde G_n, W_0 \rbrack_q, W_1 \rbrack - \lbrack W_0, \lbrack W_1, \tilde G_n \rbrack_q \rbrack
\\
&= \lbrack \tilde G_n, \lbrack W_0, W_1 \rbrack_q \rbrack
\\
&= -q \lbrack \tilde G_n, E_\delta \rbrack.
\end{align*}
\end{proof}

\begin{lemma} \label{lem:step3}
Pick $n \in \mathbb N$, and assume that {\rm (\ref{eq:X})}, {\rm (\ref{eq:Y})} hold for $n, n-1, \ldots, 1,0$. Then
\begin{align}
\lbrack W_{-n}, W_0 \rbrack = 0.
\label{eq:step3}
\end{align}
\end{lemma}
\begin{proof} The commutator $\lbrack W_{-n}, W_0\rbrack$ is equal to
\begin{align*}
&W_{-n} W_0 - W_0 W_{-n}
\\
&= \sum_{k=0}^n \frac{ E_{k\delta + \alpha_0} \tilde G_{n-k} W_0 (-1)^k q^{3k}}{(q-q^{-1})^{2k}}
-\sum_{k=0}^n \frac{ W_0 E_{k \delta + \alpha_0} \tilde G_{n-k} (-1)^k q^{3k}}{(q-q^{-1})^{2k}}
\\
&= \sum_{k=0}^n \frac{ E_{k\delta + \alpha_0} \bigl( q^{-2} W_0 \tilde G_{n-k} + (1-q^{-2}) W_{k-n}\bigr)  (-1)^k q^{3k}}{(q-q^{-1})^{2k}}
-\sum_{k=0}^n \frac{ W_0 E_{k \delta + \alpha_0} \tilde G_{n-k} (-1)^k q^{3k}}{(q-q^{-1})^{2k}}
\\
&=\sum_{k=0}^n \frac{ \lbrack W_0, E_{k \delta + \alpha_0} \rbrack_q  \tilde G_{n-k} (-1)^{k-1} q^{3k-1} }{ (q-q^{-1})^{2k}}
+
\sum_{k=0}^n \frac{ E_{k\delta + \alpha_0} W_{k-n} (-1)^k q^{3k-1}}{(q-q^{-1})^{2k-1}}
\\
&=\sum_{k=0}^n \frac{ \lbrack W_0, E_{k \delta + \alpha_0} \rbrack_q \tilde G_{n-k} (-1)^{k-1} q^{3k-1} }{ (q-q^{-1})^{2k}}
+
\sum_{k=0}^n \frac{ E_{k\delta + \alpha_0} (-1)^k q^{3k-1}}{(q-q^{-1})^{2k-1} } \sum_{\ell=0}^{n-k}  \frac{ E_{\ell \delta + \alpha_0} \tilde G_{n-k-\ell}  (-1)^\ell q^{3\ell }}{(q-q^{-1})^{2\ell}}
\\
&=\sum_{p=0}^n \frac{ \lbrack W_0, E_{p\delta + \alpha_0} \rbrack_q  \tilde G_{n-p} (-1)^{p-1} q^{3p-1} }{ (q-q^{-1})^{2p}}
+
\sum_{p=0}^n \Biggl( \sum_{k+\ell=p}   E_{k\delta + \alpha_0} E_{\ell \delta + \alpha_0} \Biggr)  \frac{ \tilde G_{n-p}   (-1)^p q^{3p-1 }}{(q-q^{-1})^{2p-1}}.
\end{align*}
\noindent By these comments
\begin{align*}
\lbrack W_{-n}, W_0\rbrack  = \sum_{p=0}^n \frac{ S_p \tilde G_{n-p} (-1)^{p-1}q^{3p-1}}{(q-q^{-1})^{2p-1}}
\end{align*}
\noindent where 
\begin{align*}
S_p & = 
 \frac{ \lbrack W_0, E_{p\delta + \alpha_0}\rbrack_q}{q-q^{-1}} 
-\sum_{k+\ell=p} E_{k\delta + \alpha_0} E_{\ell \delta + \alpha_0}  \qquad \qquad (0 \leq p \leq n).
\end{align*}  By  (\ref{eq:WEcom}) we have $S_p=0$ for $0 \leq p \leq n$. The result follows.
\end{proof}

\begin{lemma} \label{lem:step4}
Pick $n \in \mathbb N$, and assume that {\rm (\ref{eq:X})}, {\rm (\ref{eq:Y})} hold for $n, n-1, \ldots, 1,0$. Then
\begin{align}
\lbrack W_{n+1}, W_1 \rbrack = 0.
\label{eq:step4}
\end{align}
\end{lemma}
\begin{proof} The commutator $\lbrack W_{n+1}, W_1\rbrack$ is equal to
\begin{align*}
&W_{n+1} W_1 - W_1 W_{n+1}
\\&= \sum_{k=0}^n \frac{ E_{k\delta + \alpha_1} \tilde G_{n-k} W_1 (-1)^k q^{k}}{(q-q^{-1})^{2k}}
-\sum_{k=0}^n \frac{ W_1 E_{k \delta + \alpha_1} \tilde G_{n-k} (-1)^k q^{k}}{(q-q^{-1})^{2k}}
\\
&= \sum_{k=0}^n \frac{ E_{k\delta + \alpha_1} \bigl( q^{2} W_1 \tilde G_{n-k} + (1-q^{2}) W_{n-k+1}\bigr)  (-1)^k q^{k}}{(q-q^{-1})^{2k}}
-\sum_{k=0}^n \frac{ W_1 E_{k \delta + \alpha_1} \tilde G_{n-k} (-1)^k q^{k}}{(q-q^{-1})^{2k}}
\\
&=\sum_{k=0}^n \frac{ \lbrack  E_{k \delta + \alpha_1}, W_1 \rbrack_q  \tilde G_{n-k} (-1)^{k} q^{k+1} }{ (q-q^{-1})^{2k}}
-
\sum_{k=0}^n \frac{ E_{k\delta + \alpha_1} W_{n-k+1} (-1)^k q^{k+1}}{(q-q^{-1})^{2k-1}}
\\
&=\sum_{k=0}^n \frac{ \lbrack E_{k \delta + \alpha_1}, W_1 \rbrack_q \tilde G_{n-k} (-1)^{k} q^{k+1} }{ (q-q^{-1})^{2k}}
-
\sum_{k=0}^n \frac{ E_{k\delta + \alpha_1} (-1)^k q^{k+1}}{(q-q^{-1})^{2k-1} } \sum_{\ell=0}^{n-k}  \frac{ E_{\ell \delta + \alpha_1} \tilde G_{n-k-\ell}  (-1)^\ell q^{\ell }}{(q-q^{-1})^{2\ell}}
\\
&=\sum_{p=0}^n \frac{ \lbrack  E_{p\delta + \alpha_1}, W_1 \rbrack_q  \tilde G_{n-p} (-1)^{p} q^{p+1} }{ (q-q^{-1})^{2p}}
-
\sum_{p=0}^n \Biggl( \sum_{k+\ell=p}   E_{k\delta + \alpha_1} E_{\ell \delta + \alpha_1} \Biggr)  \frac{ \tilde G_{n-p}   (-1)^p q^{p+1 }}{(q-q^{-1})^{2p-1}}.
\end{align*}
\noindent By these comments
\begin{align*}
\lbrack W_{n+1}, W_1\rbrack  = \sum_{p=0}^n \frac{ T_p \tilde G_{n-p} (-1)^{p}q^{p+1}}{(q-q^{-1})^{2p-1}}
\end{align*}
\noindent where 
\begin{align*}
T_p & = 
 \frac{ \lbrack E_{p\delta + \alpha_1}, W_1 \rbrack_q}{q-q^{-1}} 
-\sum_{k+\ell=p} E_{k\delta + \alpha_1} E_{\ell \delta + \alpha_1}  \qquad \qquad (0 \leq p \leq n).
\end{align*}  By  (\ref{eq:WEcom2}) we have $T_p=0$ for $0 \leq p \leq n$. The result follows.
\end{proof}

\begin{proposition} 
\label{prop:alln}
The equations  {\rm (\ref{eq:X})}, {\rm (\ref{eq:Y})} hold in $\mathcal U$ for $n\in \mathbb N$.
\end{proposition}
\begin{proof} The proof is by induction on $n$. We assume that 
(\ref{eq:X}), (\ref{eq:Y}) hold for $n, n-1,\ldots, 1,0,$ and show that
(\ref{eq:X}), (\ref{eq:Y}) hold for $n+1$. Concerning (\ref{eq:X}),
\begin{align*}
W_{n+2} &= 
 \frac{ q W_1 \tilde G_{n+1} - q^{-1} \tilde G_{n+1} W_1 }{q-q^{-1}}               & {\mbox{ \rm by (\ref{eq:Wkp1})}}
\\
&= W_1 \tilde G_{n+1} - q^{-1} \frac{  \lbrack \tilde G_{n+1}, W_1 \rbrack }{ q-q^{-1}}
\\
&= W_1 \tilde G_{n+1} - \frac{ \lbrack \lbrack  \lbrack \tilde G_{n}, W_0 \rbrack_q, W_1 \rbrack_q, W_1\rbrack }{ (q-q^{-1})^2 (q^2-q^{-2})} 
& {\mbox{ \rm by Definition \ref{def:CE}(v)}}
\\
&= W_1 \tilde G_{n+1} - \frac{ \lbrack  \lbrack W_{-n} , W_1 \rbrack_q, W_1\rbrack }{ (q-q^{-1}) (q^2-q^{-2})} 
 & {\mbox{ \rm by (\ref{eq:Wmk})}}
\\
&= W_1 \tilde G_{n+1} - \frac{ \lbrack  \lbrack W_{-n} , W_1 \rbrack, W_1\rbrack_q }{ (q-q^{-1}) (q^2-q^{-2})}
\\
&= W_1 \tilde G_{n+1} - \frac{ \lbrack  \lbrack W_{0} , W_{n+1} \rbrack, W_1\rbrack_q }{ (q-q^{-1}) (q^2-q^{-2})}
 & {\mbox{ \rm by Lemma \ref{lem:step1}}}
\\
&= W_1 \tilde G_{n+1} - \frac{ \lbrack  \lbrack W_{0} , W_{1} \rbrack_q, W_{n+1} \rbrack }{ (q-q^{-1}) (q^2-q^{-2})}
 & {\mbox{ \rm by Lemma \ref{lem:step4}}}
\\
&= W_1 \tilde G_{n+1} + \frac{q\lbrack  E_\delta, W_{n+1} \rbrack }{ (q-q^{-1}) (q^2-q^{-2})}
 & {\mbox{ \rm by (\ref{eq:dam1})}} 
\\
&= W_1 \tilde G_{n+1} + q \sum_{k=0}^n \frac{ \lbrack E_\delta, E_{k\delta + \alpha_1} \tilde G_{n-k} \rbrack (-1)^k q^k}{(q-q^{-1})^{2k+1}(q^2-q^{-2})}
 & {\mbox{ \rm by (\ref{eq:X}) and induction}}
\\
&= W_1 \tilde G_{n+1} + q \sum_{k=0}^n \frac{ \lbrack E_\delta, E_{k\delta + \alpha_1} \rbrack \tilde G_{n-k}  (-1)^k q^k}{(q-q^{-1})^{2k+1}(q^2-q^{-2})}
 & {\mbox{ \rm by Lemma \ref{lem:step2}}}
\\
&= W_1 \tilde G_{n+1} +\sum_{k=0}^n \frac{ E_{(k+1)\delta + \alpha_1}  \tilde G_{n-k}  (-1)^{k+1} q^{k+1}}{(q-q^{-1})^{2k+2}}
 & {\mbox{ \rm by (\ref{eq:dam2})}} 
\\
&= E_{\alpha_1} \tilde G_{n+1} +\sum_{k=1}^{n+1} \frac{ E_{k\delta + \alpha_1}  \tilde G_{n+1-k}  (-1)^{k} q^{k}}{(q-q^{-1})^{2k}}
\\
&= \sum_{k=0}^{n+1} \frac{ E_{k\delta + \alpha_1}  \tilde G_{n+1-k}  (-1)^{k} q^{k}}{(q-q^{-1})^{2k}}.
\end{align*}
We have shown that (\ref{eq:X}) holds for $n+1$. Concerning (\ref{eq:Y}),
\begin{align*}
W_{-n-1} &= 
 \frac{ q \tilde G_{n+1} W_0 - q^{-1} W_0 \tilde G_{n+1}  }{q-q^{-1}}               & {\mbox{ \rm by (\ref{eq:Wmk})}}
\\
&= W_0 \tilde G_{n+1} - q \frac{  \lbrack W_0, \tilde G_{n+1} \rbrack }{ q-q^{-1}}
\\
&= W_0 \tilde G_{n+1} - q^2 \frac{ \lbrack W_0, \lbrack  W_0, \lbrack W_1, \tilde G_{n} \rbrack_q \rbrack_q \rbrack }{ (q-q^{-1})^2 (q^2-q^{-2})} 
& {\mbox{ \rm by Definition \ref{def:CE}(vi)}}
\\
&= W_0 \tilde G_{n+1} - q^2 \frac{ \lbrack W_0, \lbrack W_0, W_{n+1}  \rbrack_q \rbrack }{ (q-q^{-1}) (q^2-q^{-2})} 
 & {\mbox{ \rm by (\ref{eq:Wkp1})}}
\\
&= W_0 \tilde G_{n+1} - q^2 \frac{ \lbrack W_0, \lbrack W_0, W_{n+1}  \rbrack \rbrack_q }{ (q-q^{-1}) (q^2-q^{-2})}
\\
&= W_0 \tilde G_{n+1} - q^2 \frac{ \lbrack W_0, \lbrack W_{-n} , W_{1} \rbrack \rbrack_q }{ (q-q^{-1}) (q^2-q^{-2})}
 & {\mbox{ \rm by Lemma \ref{lem:step1}}}
\\
&= W_0 \tilde G_{n+1} - q^2 \frac{ \lbrack W_{-n}, \lbrack W_{0} , W_{1} \rbrack_q  \rbrack }{ (q-q^{-1}) (q^2-q^{-2})}
 & {\mbox{ \rm by Lemma \ref{lem:step3}}}
\\
&= W_0 \tilde G_{n+1} + q^3 \frac{\lbrack W_{-n},  E_\delta \rbrack }{ (q-q^{-1}) (q^2-q^{-2})}
 & {\mbox{ \rm by (\ref{eq:dam1})}} 
\\
&= W_0 \tilde G_{n+1} + q^3 \sum_{k=0}^n \frac{ \lbrack E_{k\delta + \alpha_0} \tilde G_{n-k}, E_{\delta}  \rbrack (-1)^k q^{3k}}{(q-q^{-1})^{2k+1}(q^2-q^{-2})}
 & {\mbox{ \rm by (\ref{eq:Y}) and induction}}
\\
&= W_0 \tilde G_{n+1} + q^3 \sum_{k=0}^n \frac{ \lbrack  E_{k\delta + \alpha_0}, E_{\delta} \rbrack \tilde G_{n-k}  (-1)^k q^{3k}}{(q-q^{-1})^{2k+1}(q^2-q^{-2})}
 & {\mbox{ \rm by Lemma \ref{lem:step2}}}
\\
&= W_0 \tilde G_{n+1} +\sum_{k=0}^n \frac{ E_{(k+1)\delta + \alpha_0}  \tilde G_{n-k}  (-1)^{k+1} q^{3k+3}}{(q-q^{-1})^{2k+2}}
 & {\mbox{ \rm by (\ref{eq:dam2})}} 
\\
&= E_{\alpha_0} \tilde G_{n+1} +\sum_{k=1}^{n+1} \frac{ E_{k\delta + \alpha_0}  \tilde G_{n+1-k}  (-1)^{k} q^{3k}}{(q-q^{-1})^{2k}}
\\
&= \sum_{k=0}^{n+1} \frac{ E_{k\delta + \alpha_0}  \tilde G_{n+1-k}  (-1)^{k} q^{3k}}{(q-q^{-1})^{2k}}.
\end{align*}
We have shown that (\ref{eq:Y}) holds for $n+1$. 
\end{proof}

\begin{lemma} 
\label{lem:ww} The following relations hold in $\mathcal U$.
For $n \in \mathbb N$,
\begin{align*}
\lbrack W_0, W_{n+1} \rbrack &= \lbrack W_{-n}, W_1 \rbrack, \qquad \qquad  \lbrack \tilde G_n, E_\delta \rbrack = 0,
\\
\lbrack W_{-n}, W_0\rbrack &= 0, \qquad \qquad \qquad \; \;\lbrack W_{n+1}, W_1\rbrack = 0.
\end{align*}
\end{lemma}
\begin{proof} By Lemmas \ref{lem:step1}--\ref{lem:step4} and Proposition \ref{prop:alln}.
\end{proof}

\begin{lemma} \label{lem:GWWG} The following relations hold in $\mathcal U$.
For $k \in \mathbb N$,
\begin{enumerate}
\item[\rm (i)] $\lbrack G_{k+1}, W_1\rbrack_q = \lbrack W_1, \tilde G_{k+1} \rbrack_q$;
\item[\rm (ii)] $\lbrack W_0, G_{k+1} \rbrack_q = \lbrack \tilde G_{k+1}, W_0 \rbrack_q$.
\end{enumerate}
\end{lemma}
\begin{proof} (i) We have
\begin{align*}
\lbrack G_{k+1}, W_1\rbrack_q - \lbrack W_1, \tilde G_{k+1} \rbrack_q
&= \biggl\lbrack \tilde G_{k+1} + \frac{\lbrack W_1, W_{-k}\rbrack }{1-q^{-2}}, W_1 \biggr\rbrack_q - \lbrack W_1, \tilde G_{k+1} \rbrack_q
\\
&=(q+q^{-1}) \lbrack \tilde G_{k+1}, W_1 \rbrack - \frac{ \lbrack \lbrack W_{-k}, W_1 \rbrack, W_1 \rbrack_q }{1-q^{-2}}
\\
&=(q+q^{-1}) \lbrack \tilde G_{k+1}, W_1 \rbrack - \frac{ \lbrack \lbrack W_{-k}, W_1 \rbrack_q, W_1 \rbrack }{1-q^{-2}}
\\
&=(q+q^{-1}) \lbrack \tilde G_{k+1}, W_1 \rbrack - \frac{ \lbrack \lbrack \lbrack \tilde G_k, W_0\rbrack_q , W_1 \rbrack_q, W_1 \rbrack }{(1-q^{-2})(q-q^{-1})}
\\
&=0.
\end{align*}
\noindent (ii) We have
\begin{align*}
\lbrack W_0, G_{k+1}\rbrack_q - \lbrack \tilde G_{k+1}, W_0 \rbrack_q
&= \biggl\lbrack W_0, \tilde G_{k+1} + \frac{\lbrack W_{k+1}, W_{0}\rbrack }{1-q^{-2}} \biggr\rbrack_q - \lbrack \tilde G_{k+1}, W_0 \rbrack_q
\\
&=(q+q^{-1}) \lbrack W_0, \tilde G_{k+1} \rbrack - \frac{ \lbrack  W_0, \lbrack W_0, W_{k+1} \rbrack \rbrack_q }{1-q^{-2}}
\\
&=(q+q^{-1}) \lbrack W_0, \tilde G_{k+1} \rbrack - \frac{ \lbrack  W_0, \lbrack W_0, W_{k+1} \rbrack_q \rbrack }{1-q^{-2}}
\\
&=(q+q^{-1}) \lbrack W_0,\tilde G_{k+1} \rbrack - \frac{ \lbrack W_0, \lbrack W_0, \lbrack W_1, \tilde G_k \rbrack_q  \rbrack_q  \rbrack }{(1-q^{-2})(q-q^{-1})}
\\
&=0.
\end{align*}

\end{proof}

\section{Generating functions}

The alternating generators of $\mathcal U$ are displayed in
(\ref{eq:WWGG}).  In the previous section we described how these generators are related to $W_0$ and $W_1$. Our next goal is to describe how the alternating generators 
are related to each other. It is convenient to use generating functions.

\begin{definition}\label{def:gf} 
\rm
We define some generating functions in an indeterminate $t$. Referring to {\rm (\ref{eq:WWGG})},
\begin{align*}
&G(t)= \sum_{n \in \mathbb N} G_n t^n, \qquad \qquad \qquad \tilde G(t) = \sum_{n\in \mathbb N} \tilde G_n t^n,
\\
&W^-(t) = \sum_{n \in \mathbb N} W_{-n} t^n, \qquad \qquad W^+(t) = \sum_{n\in \mathbb N} W_{n+1} t^n.
\end{align*}
\end{definition}

\begin{lemma}
\label{lem:WWGen}
For the algebra $\mathcal U$,
\begin{align*}
&
\frac{ \lbrack  W_0, G(t) \rbrack_q }{q-q^{-1}} = W^-(t), \qquad \qquad \frac{ \lbrack \tilde G(t), W_0 \rbrack_q }{q-q^{-1}} = W^-(t),
\\
&\lbrack W_0, W^-(t) \rbrack = 0, \qquad \qquad \qquad \frac{\lbrack W_0, W^+(t) \rbrack}{1-q^{-2}} = t^{-1}(\tilde G(t)-G(t))
\end{align*}
and
\begin{align*}
&
\frac{ \lbrack   G(t), W_1 \rbrack_q }{q-q^{-1}} = W^+(t), \qquad \qquad \frac{ \lbrack W_1, \tilde G(t) \rbrack_q }{q-q^{-1}} = W^+(t),
\\
&\lbrack W_1, W^+(t) \rbrack = 0, \qquad \qquad \qquad \frac{\lbrack W_1, W^-(t) \rbrack}{1-q^{-2}} = t^{-1}( G(t)-\tilde G(t)).
\end{align*}
\end{lemma}
\begin{proof} Use Definition
\ref{def:wwg} and Lemmas
\ref{lem:ww},
\ref{lem:GWWG}.
\end{proof}

\noindent For the rest of this section, let $s$ denote an indeterminate that commutes with $t$.
\begin{lemma} 
\label{lem:GFmain}
 For the algebra $\mathcal U$,
\begin{align*}
&\lbrack  W^-(s), W^-(t) \rbrack = 0, 
\qquad 
\lbrack W^+(s),  W^+(t) \rbrack = 0,
\\
&\lbrack  W^-(s), W^+(t) \rbrack 
+
\lbrack W^+(s), W^-(t) \rbrack = 0,
\\
&s \lbrack  W^-(s), G(t) \rbrack 
+
t \lbrack  G(s),  W^-(t) \rbrack = 0,
\\
&s \lbrack  W^-(s), {\tilde G}(t) \rbrack 
+
t \lbrack  \tilde G(s),  W^-(t) \rbrack = 0,
\\
&s \lbrack   W^+(s),  G(t) \rbrack
+
t \lbrack   G(s), W^+(t) \rbrack = 0,
\\
&s \lbrack   W^+(s), {\tilde G}(t) \rbrack
+
t \lbrack \tilde G(s), W^+(t) \rbrack = 0,
\\
&\lbrack   G(s), G(t) \rbrack = 0, 
\qquad 
\lbrack  {\tilde G}(s),  {\tilde G}(t) \rbrack = 0,
\\
&\lbrack   {\tilde G}(s), G(t) \rbrack +
\lbrack   G(s), {\tilde G}(t) \rbrack = 0
\end{align*}
and also
\begin{align*}
&\lbrack W^-(s), G(t)\rbrack_q = 
\lbrack  W^-(t), G(s)\rbrack_q,
\qquad \quad
\lbrack G(s),  W^+(t)\rbrack_q = 
\lbrack  G(t), W^+(s)\rbrack_q,
\\
&
\lbrack  {\tilde G}(s),  W^-(t)\rbrack_q = 
\lbrack  {\tilde G}(t),  W^-(s)\rbrack_q,
\qquad \quad 
\lbrack W^+(s),  {\tilde G}(t)\rbrack_q = 
\lbrack  W^+(t),  {\tilde G}(s)\rbrack_q,
\\
&t^{-1}\lbrack G(s), {\tilde G}(t)\rbrack -
s^{-1}\lbrack  G(t), {\tilde G}(s)\rbrack =
q\lbrack W^-(t), W^+(s)\rbrack_q-
q\lbrack W^-(s), W^+(t)\rbrack_q,
\\
&t^{-1}\lbrack {\tilde G}(s),   G(t)\rbrack -
s^{-1}\lbrack   {\tilde G}(t),  G(s)\rbrack =
q \lbrack  W^+(t),  W^-(s)\rbrack_q-
q\lbrack W^+(s),  W^-(t)\rbrack_q,
\\
&\lbrack G(s),  {\tilde G}(t)\rbrack_q -
\lbrack G(t), {\tilde G}(s)\rbrack_q =
qt\lbrack  W^-(t),  W^+(s)\rbrack-
qs\lbrack  W^-(s),  W^+(t)\rbrack,
\\
&\lbrack  {\tilde G}(s),   G(t)\rbrack_q -
\lbrack {\tilde G}(t),   G(s)\rbrack_q =
qt \lbrack  W^+(t), W^-(s)\rbrack-
qs\lbrack  W^+(s),  W^-(t)\rbrack.
\end{align*}
\end{lemma}
\begin{proof} 
 We refer to the generating functions $A(s,t), B(s,t),\ldots ,S(s,t)$ from Appendix A.
The present lemma asserts that for the algebra $\mathcal U$ these generating functions are all zero.
 To verify this assertion, we refer to the canonical relations in Appendix A.
We will use induction with respect to the linear order
\begin{align*}
& I(s,t), M(s,t), N(s,t), A(s,t), B(s,t), Q(s,t), D(s,t), E(s,t), F(s,t),
\\
& G(s,t),  R(s,t), S(s,t), H(s,t), K(s,t), L(s,t), P(s,t), C(s,t), J(s,t).
\end{align*}
For each element in this linear order besides $I(s,t)$, 
there exists a canonical relation that expresses the given element in terms of the previous elements in the
linear order. So in $\mathcal U$ the given element is zero, provided that in $\mathcal U$ every previous element is zero.
Note that in $\mathcal U$ we have $I(s,t)=0$  by Definition \ref{def:CE}(vii). By these comments and induction, in $\mathcal U$ every element in the linear order is zero.
We have shown that in $\mathcal U$ each of $A(s,t)$, $B(s,t), \ldots, S(s,t)$ is zero.
\end{proof}

\section{The main results}
In this section we present our main results, which are Theorem \ref{thm:Fin} and Corollary \ref{cor:altExt}.
 Recall the alternating generators
(\ref{eq:WWGG})
for $\mathcal U$.
\begin{lemma}
\label{lem:List}
The following relations hold in $\mathcal U$. For $k, \ell \in \mathbb N$ we have
\begin{align}
&
 \lbrack  W_0,  W_{k+1}\rbrack= 
\lbrack  W_{-k},  W_{1}\rbrack=
(1-q^{-2})({\tilde G}_{k+1} -  G_{k+1}),
\label{eq:n3p1vv}
\\
&
\lbrack  W_0,  G_{k+1}\rbrack_q= 
\lbrack {{\tilde G}}_{k+1},  W_{0}\rbrack_q= 
 (q-q^{-1})W_{-k-1},
\label{eq:n3p2vv}
\\
&
\lbrack G_{k+1},  W_{1}\rbrack_q= 
\lbrack  W_{1}, { {\tilde G}}_{k+1}\rbrack_q= 
(q-q^{-1}) W_{k+2},
\label{eq:n3p3vv}
\\
&
\lbrack  W_{-k},  W_{-\ell}\rbrack=0,  \qquad 
\lbrack  W_{k+1},  W_{\ell+1}\rbrack= 0,
\label{eq:n3p4vv}
\\
&
\lbrack  W_{-k},  W_{\ell+1}\rbrack+
\lbrack W_{k+1},  W_{-\ell}\rbrack= 0,
\label{eq:n3p5vv}
\\
&
\lbrack  W_{-k},  G_{\ell+1}\rbrack+
\lbrack G_{k+1},  W_{-\ell}\rbrack= 0,
\label{eq:n3p6vv}
\\
&
\lbrack W_{-k},  {\tilde G}_{\ell+1}\rbrack+
\lbrack  {\tilde G}_{k+1},  W_{-\ell}\rbrack= 0,
\label{eq:n3p7vv}
\\
&
\lbrack  W_{k+1},  G_{\ell+1}\rbrack+
\lbrack   G_{k+1}, W_{\ell+1}\rbrack= 0,
\label{eq:n3p8vv}
\\
&
\lbrack  W_{k+1},  {\tilde G}_{\ell+1}\rbrack+
\lbrack  {\tilde G}_{k+1},  W_{\ell+1}\rbrack= 0,
\label{eq:n3p9vv}
\\
&
\lbrack  G_{k+1},  G_{\ell+1}\rbrack=0,
\qquad 
\lbrack {\tilde G}_{k+1},  {\tilde G}_{\ell+1}\rbrack= 0,
\label{eq:n3p10vv}
\\
&
\lbrack {\tilde G}_{k+1},  G_{\ell+1}\rbrack+
\lbrack  G_{k+1},  {\tilde G}_{\ell+1}\rbrack= 0
\label{eq:n3p11v}
\end{align}
and also
\begin{align}
&\lbrack W_{-k}, G_{\ell}\rbrack_q = 
\lbrack W_{-\ell}, G_{k}\rbrack_q,
\qquad \quad
\lbrack G_k, W_{\ell+1}\rbrack_q = 
\lbrack G_\ell, W_{k+1}\rbrack_q,
\label{eq:ngg1}
\\
&
\lbrack \tilde G_k, W_{-\ell}\rbrack_q = 
\lbrack \tilde G_\ell, W_{-k}\rbrack_q,
\qquad \quad 
\lbrack W_{\ell+1}, \tilde G_{k}\rbrack_q = 
\lbrack W_{k+1}, \tilde G_{\ell}\rbrack_q,
\label{eq:ngg2}
\\
&\lbrack G_{k}, \tilde G_{\ell+1}\rbrack -
\lbrack G_{\ell}, \tilde G_{k+1}\rbrack =
q\lbrack W_{-\ell}, W_{k+1}\rbrack_q-
q\lbrack W_{-k}, W_{\ell+1}\rbrack_q,
\label{eq:ngg3}
\\
&\lbrack \tilde G_{k},  G_{\ell+1}\rbrack -
\lbrack \tilde G_{\ell},  G_{k+1}\rbrack =
q \lbrack W_{\ell+1}, W_{-k}\rbrack_q-
q\lbrack W_{k+1}, W_{-\ell}\rbrack_q,
\label{eq:ngg4}
\\
&\lbrack G_{k+1}, \tilde G_{\ell+1}\rbrack_q -
\lbrack G_{\ell+1}, \tilde G_{k+1}\rbrack_q =
q\lbrack W_{-\ell}, W_{k+2}\rbrack-
q\lbrack W_{-k}, W_{\ell+2}\rbrack,
\label{eq:ngg5}
\\
&\lbrack \tilde G_{k+1},  G_{\ell+1}\rbrack_q -
\lbrack \tilde G_{\ell+1},  G_{k+1}\rbrack_q =
q \lbrack W_{\ell+1}, W_{-k-1}\rbrack-
q\lbrack W_{k+1}, W_{-\ell-1}\rbrack.
\label{eq:ngg6}
\end{align}
\end{lemma}
\begin{proof} The relations (\ref{eq:n3p1vv})--(\ref{eq:n3p3vv}) are from
Definition \ref{def:wwg} and Lemmas 
\ref{lem:ww},
\ref{lem:GWWG}.
The relations (\ref{eq:n3p4vv})--(\ref{eq:ngg6}) 
follow from Definition
\ref{def:gf} and
 Lemma \ref{lem:GFmain}.
\end{proof}

\begin{theorem}
\label{thm:Fin}
The algebra $\mathcal U$ has a presentation by generators
\begin{align*}
\lbrace W_{-k} \rbrace_{k \in \mathbb N}, \qquad
\lbrace W_{k+1} \rbrace_{k \in \mathbb N}, \qquad
\lbrace G_{k+1} \rbrace_{k \in \mathbb N}, \qquad
\lbrace \tilde G_{k+1} \rbrace_{k \in \mathbb N}
\end{align*}
and the relations in Lemma \ref{lem:List}.
\end{theorem}
\begin{proof} It suffices to show that the relations in
Definition 
\ref{def:CE}
are implied by the relations in Lemma \ref{lem:List}. 
The relation (iii)  in Definition \ref{def:CE} is obtained from the equation on the left in
(\ref{eq:n3p3vv}) at $k=0$, by eliminating $G_{1}$ using
$\lbrack W_{0}, W_1 \rbrack = (1-q^{-2}) (\tilde G_{1} - G_{1})$. 
The relation (iv) in Definition \ref{def:CE} is obtained from the equation on the left in
(\ref{eq:n3p2vv}) at $k=0$, by eliminating $G_{1}$ using
$\lbrack W_0, W_{1}\rbrack = (1-q^{-2}) (\tilde G_{1} - G_{1})$.
For $k\geq 1$
the relation (v) in Definition \ref{def:CE} is obtained from the equation on the left in
(\ref{eq:n3p3vv}), by eliminating $G_{k+1}$ using
$\lbrack W_{-k}, W_1 \rbrack = (1-q^{-2}) (\tilde G_{k+1} - G_{k+1})$ 
and evaluating the result using
$\lbrack  \tilde G_k, W_0\rbrack_q = (q-q^{-1}) W_{-k}$.
For $k\geq 1$
the relation (vi) in Definition \ref{def:CE} is obtained from the equation on the left in
(\ref{eq:n3p2vv}), by eliminating $G_{k+1}$ using
$\lbrack W_0, W_{k+1}\rbrack = (1-q^{-2}) (\tilde G_{k+1} - G_{k+1})$ 
and evaluating the result using
$\lbrack W_1, \tilde G_k \rbrack_q = (q-q^{-1}) W_{k+1}$.
The relation (vii) in Definition \ref{def:CE} is from 
(\ref{eq:n3p10vv}).
The relation (i) in Definition \ref{def:CE} is obtained from $\lbrack W_0, W_{-1} \rbrack=0$, by eliminating
$W_{-1}$ using $\lbrack \tilde G_1, W_0 \rbrack_q = (q-q^{-1})W_{-1}$ and evaluating the result using
Definition \ref{def:CE}(iv). 
The relation (ii) in Definition \ref{def:CE} is obtained from $\lbrack W_1, W_{2} \rbrack=0$, by eliminating
$W_{2}$ using $\lbrack W_1, \tilde G_1 \rbrack_q = (q-q^{-1})W_{2}$ and evaluating the result using
Definition \ref{def:CE}(iii). 
\end{proof}

\noindent It is apparent from the proof of Theorem \ref{thm:Fin} that the relations in
Lemma \ref{lem:List} are redundant in the following sense.

\begin{corollary} \label{cor:redundant}
\rm
The relations in
Lemma \ref{lem:List} are implied by the relations listed in (i)--(iii) below:
\begin{enumerate}
\item[\rm (i)]  (\ref{eq:n3p1vv})--(\ref{eq:n3p3vv});
\item[\rm (ii)] (\ref{eq:n3p4vv})  with $k=0$ and $\ell=1$;
\item[\rm (iii)] the relations on the right in (\ref{eq:n3p10vv}).
\end{enumerate}
\end{corollary}
\begin{proof}  By Lemma \ref{lem:List} the relations (\ref{eq:n3p1vv})--(\ref{eq:ngg6}) are implied by
the relations in Definitions \ref{def:CE}, \ref{def:wwg}.
The relations listed in (i)--(iii) are used in the proof of Theorem \ref{thm:Fin} to obtain
the relations in Definition \ref{def:CE}. The relations listed in (i) imply the relations in Definition
\ref{def:wwg}.  The result follows.
\end{proof}

\noindent The relations in Lemma \ref{lem:List} first appeared in
\cite[Propositions~5.7, 5.10, 5.11]{alternating}.
It was observed in 
\cite[Propositions~3.1, 3.2]{baspp} and \cite[Remark~2.5]{basFMA}
 that the relations (\ref{eq:n3p1vv})--(\ref{eq:n3p11v}) imply the relations
(\ref{eq:ngg1})--(\ref{eq:ngg6}). This observation motivated the following definition.

\begin{definition}\label{def:ace}
\rm (See \cite[Definition~3.1]{altCE}.) 
Define the algebra $\mathcal U^+_q$ 
 by generators
\begin{align*}
\lbrace W_{-k} \rbrace_{k \in \mathbb N}, \qquad
\lbrace W_{k+1} \rbrace_{k \in \mathbb N}, \qquad
\lbrace G_{k+1} \rbrace_{k \in \mathbb N}, \qquad
\lbrace \tilde G_{k+1} \rbrace_{k \in \mathbb N}
\end{align*}
and the relations 
(\ref{eq:n3p1vv})--(\ref{eq:n3p11v}). The algebra $\mathcal U^+_q$ is called the
{\it alternating central extension of $U^+_q$}.
\end{definition}

\begin{corollary}\label{cor:altExt}
We have $\mathcal U =\mathcal U^+_q$.
\end{corollary}
\begin{proof}  By Theorem \ref{thm:Fin}, Corollary  \ref{cor:redundant}, and  Definition \ref{def:ace}.
\end{proof}

\begin{definition}
\rm By the {\it compact} presentation of $\mathcal U^+_q$ we mean the presentation
given in Definition \ref{def:CE}. By the {\it expanded} presentation of $\mathcal U^+_q$ we mean
the presentation given in Theorem
\ref{thm:Fin}.
\end{definition}

\begin{corollary} \label{cor:inj}
The map $\flat$ from Lemma \ref{lem:flat}
is injective.
\end{corollary}
\begin{proof} By Corollary \ref{cor:altExt} and \cite[Proposition~6.4]{altCE}.
\end{proof}

\section{The subalgebra of $\mathcal U^+_q$ generated by $\lbrace \tilde G_{k+1}\rbrace_{k\in \mathbb N}$}

Let $\tilde G$ denote the subalgebra of $\mathcal U^+_q$ generated by $\lbrace \tilde G_{k+1}\rbrace_{k\in \mathbb N}$. In this 
section we describe $\tilde G$ and its relationship to $\langle W_0, W_1 \rangle$.
\medskip

\noindent  The following notation will be useful. Let $z_1, z_2, \ldots $ denote mutually commuting
indeterminates. Let $\mathbb F\lbrack z_1, z_2, \ldots \rbrack$ denote the algebra consisting of the polynomials
in $z_1, z_2, \ldots $ that have all coefficients in $\mathbb F$. For notational convenience define $z_0=1$.

\begin{lemma} \label{lem:Gind} {\rm (See \cite[Lemma~3.5]{altCE}.)}
There exists an algebra homomorphism $ \mathcal U^+_q \to \mathbb F \lbrack z_1, z_2, \ldots \rbrack$ that sends
\begin{align*}
W_{-n} \mapsto 0, \qquad
W_{n+1} \mapsto 0, \qquad
G_n \mapsto z_n, \qquad
\tilde G_n \mapsto z_n
\end{align*}
for $n \in \mathbb N$.
\end{lemma}
\begin{proof} By Theorem \ref{thm:Fin} and the nature of the relations  in Lemma \ref{lem:List}.
\end{proof}

\begin{corollary} {\rm (See \cite[Theorem~10.2]{altCE}.)}
\label{cor:Galgind} 
The generators $\lbrace \tilde G_{k+1} \rbrace_{k\in \mathbb N}$ of $\tilde G$ are algebraically independent.
\end{corollary}
\begin{proof}
By Lemma \ref{lem:Gind} and since $\lbrace z_{k+1} \rbrace_{k\in \mathbb N}$ are algebraically independent.
\end{proof}

\noindent The following result will help us describe how $\tilde G$ is related to $\langle W_0, W_1\rangle$.
\begin{lemma}
\label{prop:Gcom}
For $n \in \mathbb N$,
\begin{align}
\tilde G_n W_1 &= W_1 \tilde G_n + \sum_{k=1}^n \frac{  E_{k \delta + \alpha_1} \tilde G_{n-k} (-1)^{k+1} q^{k+1}}{(q-q^{-1})^{2k-1}},
\label{eq:GcW1}
\\
\tilde G_n W_0 &= W_0 \tilde G_n + \sum_{k=1}^n \frac{  E_{k \delta + \alpha_0} \tilde G_{n-k} (-1)^{k} q^{3k-1}}{(q-q^{-1})^{2k-1}}.
\label{eq:GcW0}
\end{align}
\end{lemma}
\begin{proof} To obtain (\ref{eq:GcW1}), eliminate $W_{n+1}$ from 
(\ref{eq:X}) 
using (\ref{eq:Wkp1}),
and solve the resulting equation for  $\tilde G_n W_1$.
To obtain (\ref{eq:GcW0}), eliminate $W_{-n}$ from 
(\ref{eq:Y})
using (\ref{eq:Wmk}),
and solve the resulting equation for  $\tilde G_n W_0$.
\end{proof}
\noindent Shortly we will describe how $\tilde G$ is related to $\langle W_0, W_1 \rangle$. This description involves 
the center $\mathcal Z$ of $\mathcal U^+_q$. To prepare for this description, we have some comments about  $\mathcal Z$.
In \cite[Sections~5, 6]{altCE} we introduced some algebraically independent elements $Z_1, Z_2, \ldots $ that generate the algebra $\mathcal Z$. For notational convenience define $Z_0=1$.
Using $\lbrace Z_n\rbrace_{n \in \mathbb N}$ we obtain a basis for $\mathcal Z$ that is described as follows.
For $n \in \mathbb N$, a {\it partition of $n$} is a sequence $\lambda = \lbrace \lambda_i \rbrace_{i=1}^\infty$
of natural numbers such that $\lambda_i \geq \lambda_{i+1}$ for $i\geq 1$ and $n=\sum_{i=1}^\infty \lambda_i$.
The set $\Lambda_n$ consists of the partitions of $n$. Define $\Lambda = \cup_{n \in \mathbb N} \Lambda_n$. 
For $\lambda \in \Lambda$ define $Z_\lambda = \prod_{i=1}^\infty Z_{\lambda_i}$. The elements 
$\lbrace Z_\lambda\rbrace_{\lambda \in \Lambda}$ form a basis for the vector space  $\mathcal Z$.
Next we describe a grading for  $\mathcal Z$.  For $n \in \mathbb N$ let $\mathcal Z_n$ denote the subspace of
$\mathcal Z$ with basis $\lbrace Z_\lambda \rbrace_{\lambda \in \Lambda_n}$. For example $\mathcal Z_0 = \mathbb F 1$.
The sum $\mathcal Z = \sum_{n\in \mathbb N} \mathcal Z_n$ is direct. Moreover $\mathcal Z_r \mathcal Z_s\subseteq \mathcal Z_{r+s}$
for $r,s\in \mathbb N$. By these comments the subspaces $\lbrace \mathcal Z_n \rbrace_{n\in \mathbb N}$ form
a grading of $\mathcal Z$. Note that $Z_n \in \mathcal Z_n$ for $n \in \mathbb N$.
Next we describe how $\mathcal Z$ is related to $\langle W_0, W_1 \rangle$.
\begin{lemma} \label{lem:WWZ}
{\rm (See \cite[Proposition~6.5]{altCE}.)}
The multiplication map
\begin{align*}
\langle  W_0,W_1\rangle \otimes \mathcal Z &\to
	       \mathcal U^+_q 
	       \\
w \otimes z &\mapsto      wz            
\end{align*}
is an algebra isomorphism.
\end{lemma}

\noindent 
For $n \in \mathbb N$ let $\mathcal U_n$ denote the image of $\langle W_0, W_1 \rangle \otimes \mathcal Z_n $ under the multiplication map.
By construction the sum $\mathcal U^+_q=\sum_{n\in \mathbb N} \mathcal U_n$ is direct.
\medskip

\noindent In the next two lemmas  we describe how $\tilde G$ is related to $\mathcal Z$.

\begin{lemma} \label{lem:GnZn}  {\rm (See \cite[Lemmas~3.6, 5.9]{altCE}.)} 
For $n \in \mathbb N$,
\begin{align*}
\tilde G_n \in \sum_{k=0}^{n} \langle W_0, W_1 \rangle  Z_k,
\qquad \qquad 
\tilde G_n - Z_n \in \sum_{k=0}^{n-1} \langle W_0, W_1 \rangle Z_k.
\end{align*}
\end{lemma}
\noindent For $\lambda \in \Lambda$ define $\tilde G_\lambda = \prod_{i=1}^\infty \tilde G_{\lambda_i}$.
By Corollary \ref{cor:Galgind} the elements $\lbrace \tilde G_\lambda \rbrace_{\lambda \in \Lambda}$ form a basis for the
vector space $\tilde G$.
\begin{lemma} \label{lem:GGZ}
For $n \in \mathbb N$ and $\lambda \in \Lambda_n$,
\begin{align*}
\tilde G_\lambda \in \sum_{k=0}^{n} \mathcal U_k,
\qquad \qquad \qquad 
\tilde G_\lambda - Z_\lambda \in \sum_{k=0}^{n-1} \mathcal U_k.
\end{align*}
\end{lemma}
\begin{proof} By Lemma  \ref{lem:GnZn} and our comments above Lemma
\ref{lem:WWZ}
about the grading of $\mathcal Z$.
\end{proof}
\noindent Next we describe how $\tilde G$ is related to $\langle W_0, W_1 \rangle$.
\begin{proposition}
\label{prop:prod} The multiplication map
\begin{align*}
\langle  W_0,W_1\rangle \otimes \tilde G &\to
	       \mathcal U^+_q 
	       \\
w \otimes g &\mapsto      wg            
\end{align*}
is an isomorphism of vector spaces.
\end{proposition}
\begin{proof} The multiplication map is $\mathbb F$-linear. The multiplication map is surjective by
Lemma \ref{prop:Gcom} and since $\mathcal U^+_q$ is generated by $W_0$,  $W_1$, $\tilde G$.
We now show that the multiplicaton map is injective.
Consider a vector $v \in \langle  W_0, W_1\rangle \otimes \tilde G $ that is sent to zero by  the multiplication map. We show that $v=0$.
Write
$v=\sum_{\lambda \in \Lambda} a_\lambda \otimes \tilde G_\lambda$, where 
 $a_\lambda \in \langle W_0, W_1 \rangle $ for $\lambda \in \Lambda $ and $a_\lambda =0$ for all but finitely many $\lambda \in \Lambda$.
To show that $v=0$, we must show that $a_\lambda=0$ for all $\lambda \in \Lambda$. 
Suppose that there exists $\lambda \in \Lambda$ such that $a_\lambda \not= 0$.
 Let $C$ denote the set of natural numbers  $m$ such that 
  $\Lambda_m$ contains a partition $\lambda$
 with $a_\lambda \not=0$. The set $C$ is nonempty and finite. Let $n$ denote the maximal element of $C$.
 By construction $\sum_{\lambda \in \Lambda_n} a_\lambda \otimes Z_\lambda$ is nonzero. By Lemma \ref{lem:WWZ},
\begin{align}
\sum_{\lambda \in \Lambda_n} a_\lambda Z_\lambda \not=0.
\label{eq:nzer}
\end{align}
By construction
\begin{align}
\label{eq:kk}
0= \sum_{\lambda \in \Lambda} a_\lambda \tilde G_\lambda =
\sum_{k=0}^{n} \sum_{\lambda \in \Lambda_k} a_\lambda \tilde G_\lambda
=
 \sum_{\lambda \in \Lambda_n} a_\lambda \tilde G_\lambda + \sum_{k=0}^{n-1} \sum_{\lambda \in \Lambda_k} a_\lambda \tilde G_\lambda.
\end{align}
Using (\ref{eq:kk}),
\begin{align}
\label{eq:mainarg}
\sum_{\lambda \in \Lambda_n} a_\lambda Z_\lambda = 
\sum_{\lambda \in \Lambda_n} a_\lambda (Z_\lambda - \tilde G_\lambda) - \sum_{k=0}^{n-1} \sum_{\lambda \in \Lambda_k} a_\lambda \tilde G_\lambda.
\end{align}
The left-hand side of (\ref{eq:mainarg})  is contained in $\mathcal U_n$.
 By 
 Lemma \ref{lem:GGZ}
 the right-hand side of (\ref{eq:mainarg}) 
is contained in $\sum_{k=0}^{n-1} \mathcal U_k$.
The subspaces $\mathcal U_n$ and
$\sum_{k=0}^{n-1} \mathcal U_k$ have zero intersection because the sum $\sum_{k=0}^n \mathcal U_k$ is direct.
This contradicts (\ref{eq:nzer}), so $a_\lambda=0$ for $\lambda \in \Lambda$. Consequently $v=0$, as desired. We have shown that
 the multiplication
 map is injective. By the above comments the multiplication map is an isomorphism of vector spaces.
\end{proof}

\section{Acknowledgment} The author thanks Pascal Baseilhac for many conversations about $U^+_q$ and its central extension $\mathcal U^+_q$.
The author thanks Kazumasa Nomura for giving this paper a close reading and offering many valuable comments.

\section{Appendix A}

Recall the algebra $\mathcal U$ from Definition \ref{def:CE}. In
 this appendix we list some relations that hold in $\mathcal U$. We will define an algebra $\mathcal U^\vee$ that is a
homomorphic preimage of $\mathcal U$. All the results in this appendix are about $\mathcal U^\vee$.
\medskip

\noindent 
Define the algebra $\mathcal U^\vee$ by generators 
\begin{align*}
\lbrace W_{-k}\rbrace_{k\in \mathbb N}, \qquad  \lbrace  W_{k+1}\rbrace_{k\in \mathbb N},\qquad  
 \lbrace G_{k+1}\rbrace_{k\in \mathbb N},
\qquad
\lbrace {\tilde G}_{k+1}\rbrace_{k\in \mathbb N}
\end{align*}
 and the following relations. For $k \in \mathbb N$,
\begin{align}
&
 \lbrack W_0, W_{k+1}\rbrack= 
\lbrack W_{-k}, W_{1}\rbrack=(1-q^{-2})
(\tilde G_{k+1} -  G_{k+1}),
\label{eq:3p1App}
\\
&
\lbrack W_0, G_{k+1}\rbrack_q= 
\lbrack {\tilde G}_{k+1}, W_{0}\rbrack_q= 
(q-q^{-1}) W_{-k-1},
\label{eq:3p2App}
\\
&
\lbrack  G_{k+1},  W_{1}\rbrack_q= 
\lbrack  W_{1}, {\tilde G}_{k+1}\rbrack_q= 
(q-q^{-1}) W_{k+2},
\label{eq:3p3App}
\\
&
\lbrack  W_{0}, W_{-k}\rbrack=0,  \qquad 
\lbrack W_{1},  W_{k+1}\rbrack= 0.
\label{eq:3p4App}
\end{align}
\noindent For notational convenience, define $ G_0=1$ and ${\tilde G}_0=1$.
\medskip

\noindent For $\mathcal U^\vee$ we define the generating functions $W^-(t)$, $W^+(t)$, $G(t)$, ${\tilde G}(t)$ as in Definition \ref{def:gf}.
In terms of these generating functions, the relations \eqref{eq:3p1App}--\eqref{eq:3p4App} become the relations in Lemma \ref{lem:WWGen}.
 Let $s$ denote an indeterminate that commutes with $t$.
Define
\begin{align*}
A(s,t)&=\lbrack  W^-(s), W^-(t) \rbrack ,
\\
B(s,t)&=\lbrack W^+(s),  W^+(t) \rbrack,
\\
C(s,t)&=\lbrack  W^-(s), W^+(t) \rbrack 
+
\lbrack W^+(s), W^-(t) \rbrack,
\\
D(s,t)&=s \lbrack  W^-(s), G(t) \rbrack 
+
t \lbrack  G(s),  W^-(t) \rbrack ,
\\
E(s,t)&=s \lbrack  W^-(s), {\tilde G}(t) \rbrack 
+
t \lbrack  \tilde G(s),  W^-(t) \rbrack,
\\
F(s,t)&=s \lbrack   W^+(s),  G(t) \rbrack
+
t \lbrack   G(s), W^+(t) \rbrack,
\\
G(s,t)&=s \lbrack   W^+(s), {\tilde G}(t) \rbrack
+
t \lbrack \tilde G(s), W^+(t) \rbrack,
\\
H(s,t)&=\lbrack   G(s), G(t) \rbrack, 
\\
I(s,t)&=\lbrack  {\tilde G}(s),  {\tilde G}(t) \rbrack,
\\
J(s,t)&=\lbrack   {\tilde G}(s), G(t) \rbrack +
\lbrack   G(s), {\tilde G}(t) \rbrack
\end{align*}
and also
\begin{align*}
K(s,t)&=\lbrack W^-(s), G(t)\rbrack_q -
\lbrack  W^-(t), G(s)\rbrack_q,
\\
L(s,t)&=\lbrack G(s),  W^+(t)\rbrack_q -
\lbrack  G(t), W^+(s)\rbrack_q,
\\
M(s,t)&=\lbrack  {\tilde G}(s),  W^-(t)\rbrack_q -
\lbrack  {\tilde G}(t),  W^-(s)\rbrack_q,
\\
N(s,t)&=\lbrack W^+(s),  {\tilde G}(t)\rbrack_q -
\lbrack  W^+(t),  {\tilde G}(s)\rbrack_q,
\\
P(s,t)&=t^{-1}\lbrack G(s), {\tilde G}(t)\rbrack -
s^{-1}\lbrack  G(t), {\tilde G}(s)\rbrack -
q\lbrack W^-(t), W^+(s)\rbrack_q+
q\lbrack W^-(s), W^+(t)\rbrack_q,
\\
Q(s,t)&=t^{-1}\lbrack {\tilde G}(s),   G(t)\rbrack -
s^{-1}\lbrack   {\tilde G}(t),  G(s)\rbrack -
q \lbrack  W^+(t),  W^-(s)\rbrack_q+
q\lbrack W^+(s),  W^-(t)\rbrack_q,
\\
R(s,t)&=\lbrack G(s),  {\tilde G}(t)\rbrack_q -
\lbrack G(t), {\tilde G}(s)\rbrack_q -
qt\lbrack  W^-(t),  W^+(s)\rbrack+
qs\lbrack  W^-(s),  W^+(t)\rbrack,
\\
S(s,t)&=\lbrack  {\tilde G}(s),   G(t)\rbrack_q -
\lbrack {\tilde G}(t),   G(s)\rbrack_q -
qt \lbrack  W^+(t), W^-(s)\rbrack+
qs\lbrack  W^+(s),  W^-(t)\rbrack.
\end{align*}
By linear algebra,
\begin{align}
\label{eq:CextraA}
C(s,t) &= \frac{ (q+q^{-1})(P(s,t)+Q(s,t))-(s^{-1}+t^{-1})(R(s,t)+S(s,t))}{(q^2-s^{-1}t)(q^2 -st^{-1})q^{-1}},
\\
\label{eq:JextraA}
J(s,t) &= \frac{ (q+q^{-1})(R(s,t)+S(s,t))-(s+t)(P(s,t)+Q(s,t))}{(q^2-s^{-1}t)(q^2-st^{-1})q^{-2}}.
\end{align}
Using Lemma \ref{lem:WWGen} we routinely obtain
\begin{align*}
\lbrack W_0, A(s,t)\rbrack &=0,
 \qquad \qquad \qquad \qquad \qquad
\frac{ \lbrack W_0, B(s,t) \rbrack}{1-q^{-2}} = \frac{G(s,t) - F(s,t)}{st},
\\
\frac{ \lbrack W_0, C(s,t) \rbrack }{1-q^{-2}} &= \frac{E(s,t)-D(s,t)}{st},
\qquad \quad
\frac{\lbrack W_0,D(s,t) \rbrack_q }{q-q^{-1}} = (s+t)A(s,t),
\\
\frac{\lbrack E(s,t), W_0 \rbrack_q }{q-q^{-1}} &= (s+t)A(s,t),
\qquad \qquad 
\frac{\lbrack W_0, F(s,t) \rbrack_q }{1-q^{-2}} = S(s,t)-(q+q^{-1}) H(s,t),
\\
\frac{\lbrack G(s,t), W_0\rbrack_q }{1-q^{-2}} &= S(s,t)-(q+q^{-1})I(s,t),
\qquad 
\frac{\lbrack W_0, H(s,t) \rbrack_{q^2}}{q-q^{-1}} = K(s,t),
\\
\frac{\lbrack I(s,t), W_0 \rbrack_{q^2} }{q-q^{-1}} &= M(s,t),
\qquad 
 \qquad \qquad \qquad 
\frac{ \lbrack W_0, J(s,t) \rbrack }{q-q^{-1}} = M(s,t) - K(s,t)
\end{align*}
\noindent and 
\begin{align*}
\frac{ \lbrack W_0, K(s,t) \rbrack_q }{ q^2-q^{-2}} &= A(s,t),
\qquad \qquad 
\frac{\lbrack W_0, L(s,t) \rbrack_q }{q-q^{-1}} = P(s,t)-(s^{-1} + t^{-1})H(s,t),
\\
\frac{ \lbrack M(s,t), W_0 \rbrack_q }{ q^2-q^{-2}} &= A(s,t),
\qquad \qquad 
\frac{\lbrack N(s,t), W_0\rbrack_q }{q-q^{-1}} = Q(s,t)-(s^{-1} + t^{-1})I(s,t),
\\
\frac{ \lbrack P(s,t), W_0 \rbrack}{q-q^{-1}} &= (s^{-1} + t^{-1}) K(s,t) - (q+q^{-1})s^{-1} t^{-1} E(s,t),
\\
\frac{ \lbrack W_0, Q(s,t) \rbrack}{q-q^{-1}} &= (s^{-1} + t^{-1}) M(s,t) - (q+q^{-1})s^{-1} t^{-1} D(s,t),
\\
\frac{\lbrack W_0, R(s,t) \rbrack }{q-q^{-1}} &= (s^{-1} + t^{-1})(E(s,t)-D(s,t)),
\qquad 
\frac{ \lbrack W_0, S(s,t) \rbrack }{ q^2-q^{-2}} = M(s,t) -K(s,t)
\end{align*}
\noindent and
\begin{align*}
\frac{ \lbrack W_1, A(s,t) \rbrack}{1-q^{-2}} &= \frac{D(s,t) - E(s,t)}{st},
 \qquad \qquad \qquad 
\lbrack W_1, B(s,t)\rbrack =0,
\\
\frac{ \lbrack W_1, C(s,t) \rbrack }{1-q^{-2}} &= \frac{F(s,t)-G(s,t)}{st},
\qquad 
\frac{\lbrack D(s,t), W_1\rbrack_q }{1-q^{-2}} = R(s,t)-(q+q^{-1})H(s,t),
\\
\frac{\lbrack W_1, E(s,t) \rbrack_q }{1-q^{-2}} &= R(s,t)-(q+q^{-1}) I(s,t),
\qquad 
\frac{\lbrack F(s,t), W_1 \rbrack_q }{q-q^{-1}} = (s+t)B(s,t),
\\
\frac{\lbrack W_1,G(s,t) \rbrack_q }{q-q^{-1}} &= (s+t)B(s,t),
 \qquad \qquad \qquad \quad
\frac{\lbrack H(s,t), W_1 \rbrack_{q^2} }{q-q^{-1}} = L(s,t),
\\
\frac{\lbrack W_1, I(s,t) \rbrack_{q^2}}{q-q^{-1}} &= N(s,t),
 \qquad \qquad \qquad \qquad \quad
\frac{ \lbrack W_1, J(s,t) \rbrack }{q-q^{-1}} = L(s,t) - N(s,t)
\end{align*}
\noindent and 
\begin{align*}
\frac{\lbrack K(s,t), W_1\rbrack_q }{q-q^{-1}} &= P(s,t)-(s^{-1} + t^{-1})H(s,t),
\qquad \qquad 
\frac{ \lbrack L(s,t), W_1 \rbrack_q }{ q^2-q^{-2}} = B(s,t),
\\
\frac{\lbrack W_1, M(s,t) \rbrack_q }{q-q^{-1}} &= Q(s,t)-(s^{-1} + t^{-1})I(s,t),
\qquad \qquad 
\frac{ \lbrack W_1, N(s,t) \rbrack_q }{ q^2-q^{-2}} = B(s,t),
\\
\frac{ \lbrack W_1, P(s,t) \rbrack}{q-q^{-1}} &= (s^{-1} + t^{-1}) L(s,t) - (q+q^{-1})s^{-1} t^{-1} G(s,t),
\\
\frac{ \lbrack Q(s,t), W_1 \rbrack}{q-q^{-1}} &= (s^{-1} + t^{-1}) N(s,t) - (q+q^{-1})s^{-1} t^{-1} F(s,t),
\\
\frac{ \lbrack W_1, R(s,t) \rbrack }{ q^2-q^{-2}} &= L(s,t) -N(s,t),
 \qquad
\frac{\lbrack W_1, S(s,t) \rbrack }{q-q^{-1}} = (s^{-1} + t^{-1})(F(s,t)-G(s,t)).
\end{align*}

\noindent
We just listed 38 relations, including 
(\ref{eq:CextraA}), (\ref{eq:JextraA}). These 38 relations are 
called {\it canonical}.

\bigskip

\noindent Paul Terwilliger \hfil\break
\noindent Department of Mathematics \hfil\break
\noindent University of Wisconsin \hfil\break
\noindent 480 Lincoln Drive \hfil\break
\noindent Madison, WI 53706-1388 USA \hfil\break
\noindent email: {\tt terwilli@math.wisc.edu }\hfil\break

\end{document}